\documentclass{amsart}

\usepackage{amsmath,amssymb,amsthm}

\newtheorem{prop}{Proposition}[section]
\newtheorem{thm}[prop]{Theorem}
\newtheorem{lem}[prop]{Lemma}

\theoremstyle{definition}

\newtheorem{rem}[prop]{Remark}

\newtheorem*{ack}{Acknowledgement}

\def\co{\colon\thinspace}

\newcommand{\C}{\mathbb C}

\newcommand{\rmd}{\mathrm d}

\newcommand{\rme}{\mathrm e}

\newcommand{\R}{\mathbb R}

\newcommand{\ra}{\rightarrow}

\DeclareMathOperator{\fs}{\mathrm{FS}}

\begin{document}

\author{Gabriele Benedetti}
\address{Mathematisches Institut, Westf\"alische Wilhelms-Universit\"at M\"unster, Einsteinstr. 62,
D-48149 M\"unster, Germany}
\email{benedett@uni-muenster.de, kai.zehmisch@uni-muenster.de}
\author{Kai Zehmisch}

\title[Periodic orbits for magnetic systems]
{On the existence of periodic orbits for magnetic systems on the two-sphere}

\date{May 5, 2015}

\begin{abstract}
We prove that there exist periodic orbits on almost all compact regular energy
levels of a Hamiltonian function defined on a twisted cotangent bundle over
the two-sphere. As a corollary, given any Riemannian two-sphere and a magnetic
field on it, there exists a closed magnetic geodesic for almost all
kinetic energy levels.
\end{abstract}

\thanks{GB is supported by the SFB 878 - Groups, Geometry and Actions.
KZ is partially supported by DFG grant ZE 992/1-1.}

\maketitle


\section{Introduction\label{sec:intro}}

We consider the cotangent bundle of the $2$-sphere
together with the bundle projection $\pi\co T^*S^2\ra S^2$.
The Liouville $1$-form $\lambda$ on $T^*S^2$
is defined by $\lambda_u=u\circ T\pi$
for all co-vectors $u\in T^*S^2$.
The differential $\rmd\lambda$ provides $T^*S^2$
with the canonical symplectic structure.
For any $2$-form $\sigma$ on $S^2$,
which is closed for dimensional reasons,
the so-called {\bf twisted symplectic form}
\[
\omega_{\sigma}=\rmd\lambda+\pi^*\sigma
\]
is defined.
The equations of motion
of the Hamiltonian system
given by $\omega_{\sigma}$
and the sum of kinetic and potential energy
describes the motion of a charged particle on $S^2$
under the influence of
the force field induced by the potential and
the magnetic field induced by $\sigma$.
We call $\sigma$ the {\bf magnetic form}
and $\big(T^*S^2,\omega_{\sigma}\big)$
a {\bf twisted cotangent bundle}.

Let $H$ be a Hamiltonian function
on a twisted cotangent bundle over $S^2$.
For any compact regular energy surface $\{H=E\}$
we find $\varepsilon>0$
such that all energy levels of energy
between $E-\varepsilon$ and
$E+\varepsilon$
are compact and regular.
We say that $\{H=E\}$
has the {\bf almost existence property}
if for all such $\varepsilon>0$
the set of energy values in
$(E-\varepsilon,E+\varepsilon)$,
for which the corresponding energy surface of $H$
admits a periodic solution,
has measure $2\varepsilon$.

\begin{thm}
  \label{thm:main}
 Any compact regular energy surface
 of a Hamiltonian function
 on a twisted cotangent bundle over $S^2$
 has the almost existence property.
\end{thm}

By passage to the double cover
the theorem implies
the almost existence property
for Hamiltonian functions on twisted cotangent bundles
over the real projective plane.

We remark that the third homology group
of $T^*S^2$ vanishes.
Therefore,
each compact connected orientable hypersurface $M$
in $T^*S^2$ is the boundary of a relatively compact domain $D$.
By the results of
Hofer-Zehnder \cite{hz87,hz94}
and Struwe \cite{str90}
the almost existence property for $M$,
viewed as the regular level set of a function,
follows provided that the Hofer-Zehnder capacity of $D$
is finite.
For an alternative argument we refer to
or Macarini-Schlenk \cite{msch05},
which only requires finiteness of the Hofer-Zehnder capacity
of a neighbourhood of $M$.
Therefore, Theorem \ref{thm:main} is implied by:

\begin{thm}
\label{thm:hzcapfin}
 Each open and relatively compact subset of
 $\big(T^*S^2,\omega_{\sigma}\big)$
 has finite Hofer-Zehnder capacity.
\end{thm}

We remark that the Hofer-Zehnder capacity
of any twisted cotangent bundle (of dimension $2n$)
is infinite as the cotangent bundle of $\R^n$
embeds symplectically as observed by Lu \cite[Theorem D(ii)]{lu98}.

Hamiltonian systems on twisted cotangent bundles
of closed manifolds were first introduced by Arnol'd in \cite{arn61},
where he considers as Hamiltonian function
the kinetic energy
with respect to some Riemannian metric.
These systems are significant
since they model many physical phenomena
in classical mechanics:
a charged particle under the influence of a magnetic force,
a rigid body in an axially symmetric field.
At the beginning of the Eighties,
Novikov got interested in such systems
from the point of view of
periodic trajectories,
see \cite{nosh81,nov82}.
This direction of research
was later continued by Ta\u\i manov
in a series of works,
see \cite{tai92} for a survey.

Starting with the work of Ginzburg \cite{gin87}
and Polterovich \cite{pol98}
symplectic approaches to the existence of periodic orbits proved also to be effective,
see also Macarini \cite{mac04}.
Among the tools used for this purpose
the Hofer-Zehnder capacity,
with which we are concerned here,
plays a predominant role.
In the setting of twisted cotangent bundles
it was studied by Lu \cite[Theorem E]{lu98},
Ginzburg-G\"urel \cite[Section 2.3]{gigu04},
Schlenk \cite[Section 3.3]{sch06},
Frauenfelder-Schlenk \cite[Theorem 4.B]{fsch07},
Cieliebak-Frauenfelder-Paternain \cite{cfp10},
and Irie \cite{iri11,iri14}.
However,
almost existence for high energy levels
and non-exact magnetic forms
was left open for the two-sphere.

Finally, we would like to mention
the groundbreaking work of Contreras \cite{con06}
on autonomous Lagrangian systems,
which paved the way to tremendous advancements
for the problem of closed orbits on twisted cotangent bundles,
see Merry \cite{mer10} and \cite{assben14}.
For multiplicity results when the base manifold is a surface
see also Abbondandolo-Macarini-Paternain \cite{amp15}
as well as \cite{ammp14,assben15}.

The idea of the proof of Theorem \ref{thm:hzcapfin}
is to show that the twisted symplectic form
can be interpolated to the canonical 
Liouville symplectic form.
A compact neighbourhood of the zero-section
that supports the interpolation in turn
embeds symplectically into $S^2\times S^2$
provided with a split symplectic form.
The Hofer-Zehnder capacity of
such a symplectic manifold is finite,
so that the result will follows from the monotonicity
property of the capacity.


\section{The cut off lemma\label{sec:cutoff}}

That $\omega_{\sigma}$ is indeed a symplectic form
follows from the local description
\[
\omega_{\sigma}=
\rmd p_i\wedge\rmd q^i+
\sigma_{jk}(q)\rmd q^j\wedge\rmd q^k
\]
with respect to canonical
$(q,p)$-coordinates.
In fact, the volume forms
$\omega_{\sigma}\wedge\omega_{\sigma}$
and $\rmd\lambda\wedge\rmd\lambda$
are the same.
Moreover, the canonical Liouville vector field $Y$,
which equals $p_i\partial_{p_i}$ locally,
preserves the twisting $\pi^*\sigma$.
In fact, $i_Y\pi^*\sigma=0$.

\begin{lem}
\label{lem:cutoff}
 For any compact subset $K$ of $T^*S^2$
 there exists a symplectic form on $T^*S^2$ 
 that coincides with $\omega_{\sigma}$
 in a neighbourhood of $K$
 and with $\rmd\lambda$ in a neighbourhood
 of the end.
\end{lem}

\begin{proof}
 We provide $S^2$ with the metric induced
 by the inclusion into $\R^3$.
 This allows us to identify the unit co-sphere bundle
 of $S^2$ with $\R P^3$.
 The restriction of $\lambda$ to the tangent bundle
 of $\R P^3$ induces a contact form $\alpha$ on $\R P^3$.
 The set of all non-zero co-vectors $u\in T^*S^2$
 can be identified with the symplectisation
 \[
 \big(
 \R\times\R P^3,\rmd(\rme^t\alpha)
 \big)
 \]
 symplectically via the map
 \[
 u\longmapsto
 \big(
 \ln|u|,u/|u|\,
 \big),
 \]
 which sends the flow lines of $Y$
 to the flow lines of $\partial_t$,
 $t\in\R$,
 and hence $\lambda$ to $\rme^t\alpha$.
 
 We claim that $\pi^*\sigma$ restricted to
 the symplectisation of $\R P^3$
 has a primitive $\tau$
 that is a pull back along the inclusion
 $\R P^3\equiv\{0\}\times\R P^3\subset\R\times\R P^3$.
 Indeed, write
 \[
 \pi^*\sigma=
 a\,\rmd t\wedge\gamma_t+\eta_t
 \]
 for the image of $\pi^*\sigma$ on $\R\times\R P^3$,
 where $a$ is a function on $\R\times\R P^3$
 and $\gamma_t$, resp., $\eta_t$ is a $t$-parameter family
 of $1$-, resp., $2$-forms on  $\R P^3$.
 As observed at the beginning of the section
 inner multiplication of the twisting $\pi^*\sigma$
 by $Y\equiv\partial_t$
 yields zero, so that
 $a\gamma_t=0$ and, hence,
 $\pi^*\sigma=\eta_t$ on the symplectisation.
 Therefore,
 we get $0=\dot\eta_t\wedge\rmd t+\rmd\eta_t$
 because $\sigma$ is closed,
 where the dot indicates the time derivative.
 This implies $\dot\eta_t=0$ again by using $\partial_t$,
 and hence $\rmd\eta_t=0$.
 In other words, $\pi^*\sigma=\eta$
 for a closed $2$-form $\eta$ on $\R P^3$,
 which must have a primitive $\tau$
 because the second de Rham cohomology of
 $\R P^3$ vanishes.
 In conclusion, $\pi^*\sigma=\rmd\tau$ on $\R\times\R P^3$.
 
 Let $t_0$ be a sufficiently large real number.
 Let $f$ be a cut off function on $\R$,
 that is identically $1$ on $(-\infty,t_0]$, resp., $0$
 on $[t_0+R,\infty)$ for some positive real number $R$,
 and has derivative $\dot f$
 with values in the interval $(-2/R,0]$.
 In order to prove the lemma
 we will show that
 \[
 \rmd(\rme^t\alpha)+\rmd(f\tau)
 \]
 is a symplectic form for sufficiently large $R$.
 Because the volume forms induced by
 $\rmd(\rme^t\alpha)$ and the twisted
 $\rmd(\rme^t\alpha)+\rmd\tau$
 coincide
 their difference
 $-2\rme^t\rmd t\wedge\alpha\wedge\rmd\tau$
 must vanish.
 Hence
 \[
 \alpha\wedge\rmd\tau=0.
 \]
 Therefore,
 the square of $\rmd(\rme^t\alpha)+\rmd(f\tau)$
 equals
 \[
 2\rmd t\wedge
 \Big(
 \rme^{2t}\,\alpha\wedge\rmd\alpha+
 \dot f\rme^t\;\tau\wedge\rmd\alpha+
 \dot ff\;\tau\wedge\rmd\tau
 \Big).
 \]
 This is indeed a volume form provided that
 the $3$-forms
 \[
 \rme^t\,\alpha\wedge\rmd\alpha+
 \dot f
 \Big(
 \tau\wedge\rmd\alpha+
 f\rme^{-t}\;\tau\wedge\rmd\tau
 \Big)
 \]
 on $\R P^3$ are positive for all $t$.
 For $R$ sufficiently large this is the case
 because $|\dot f|$ is bounded by $2/R$
 and $\dot f$ has support in $[t_0,t_0+R]$.
\end{proof}

\begin{rem}
 The existence of a primitive of $\eta$
 follows by \cite[Lemma 12.6]{fsch07} as well.
 In fact,
 in the case of a general connected base manifold $Q$
 the Gysin sequence implies
 that $\eta$ admits a primitive
 if and only if the cohomology class of $\eta$
 is a multiple of the Euler class of $Q$.
 In other words, $\eta$ admits a primitive
 if either the magnetic form is exact
 or $Q$ is a surface different from a torus,
 cf.\ \cite[Remark 12.7]{fsch07}.
\end{rem}


\section{Capacity bounds\label{sec:capbounds}}

Let $U$ be an open relatively compact subset
of the twisted cotangent bundle
$\big(T^*S^2,\omega_{\sigma}\big)$.
In view of Lemma \ref{lem:cutoff}
we can cut off the twisting outside a neighbourhood
of $\overline{U}$
so that the resulting symplectic form is standard
on the complement of a co-disc bundle $V$
of sufficiently large radius.

We provide $\C P^1\times\C P^1$
with the split symplectic form
$C\big(\omega_{\fs}\oplus\omega_{\fs}\big)$,
where $\omega_{\fs}$ denotes the Fubini-Study form,
so that the anti-diagonal is a Lagrangian $2$-sphere.
The co-disc bundle $V$
equipped with the Liouville symplectic form
embeds symplectically
for $C>0$ sufficiently large.
Pushing forward the cut off twisted symplectic form
constructed in Lemma \ref{lem:cutoff}
we obtain a symplectic manifold $(W,\omega)$
diffeomorphic to $\C P^1\times\C P^1$,
so that $(U,\omega_{\sigma})$
embeds symplectically.
Therefore,
in order to prove Theorem \ref{thm:hzcapfin}
it suffices to show that $(W,\omega)$
has finite Hofer-Zehnder capacity.

Observe that the diagonal of $\C P^1\times\C P^1$
is a symplectically embedded $2$-sphere
of self-intersection $2$ in $(W,\omega)$.
Because the intersection form of $\C P^1\times\C P^1$
is even
the symplectic manifold $(W,\omega)$
is minimal.
By the classification result
\cite[Corollary 1.6]{mcd90}
of McDuff, $(W,\omega)$
is symplectomorphic to $\C P^1\times\C P^1$
provided with a split symplectic form
$a\,\omega_{\fs}\oplus b\,\omega_{\fs}$,
$a,b>0$,
cf.\ \cite[Remark (1) on p.\ 6]{lamc96}.
Using Lu's finiteness result
\cite[Theorem 1.21]{lu06},
which is based on work of
Hofer-Viterbo \cite{hovi92},
Floer-Hofer-Salamon \cite{fhs95}, and
Liu-Tian \cite{liutia00} (cf.\ McDuff-Slimowitz
\cite{mcsl01} and \cite{mcsa04}),
we obtain $(a+b)\pi$ for the variant of the
Hofer-Zehnder capacity of $(W,\omega)$
that is defined by
Frauenfelder-Ginzburg-Schlenk
in \cite{fgsch05}.
According to \cite[Appendix B]{fgsch05}
almost existence still is implied.


\begin{ack}
  We thank Alberto Abbondandolo, Peter Albers, and
  Hansj\"org Geiges for their interest in our work
  as well as Matthew Strom Borman,
  Jungsoo Kang and Stefan Suhr
  for enlightening discussions.
\end{ack}


\end{document}